\documentclass{article}

\usepackage{amsmath}
\usepackage{amsthm}
\usepackage{amsfonts}
\usepackage[latin5]{inputenc}
\usepackage{color}

\newenvironment{BicEq}{%
\addtocounter{equation}{-1}

\begin{equation}}
{\end{equation}}

\begin{document}

\newcounter{Remarkcounter}
\newtheorem{theorem}{Theorem}
\newtheorem{Prop}{Proposition}
\newtheorem{Corol}{Corollary}
\newtheorem{Lemma}{Lemma}
\newtheorem{Remark}[Remarkcounter]{Remark}
\newtheorem{Example}[Remarkcounter]{Example}

\title{COMPLETE CLASSIFICATION OF BICONSERVATIVE HYPERSURFACES WITH DIAGONALIZABLE SHAPE OPERATOR IN THE MINKOWSKI 4-SPACE}
\author{Yu Fu\footnote{ School of Mathematics and Quantitative Economics, Dongbei University of Finance and Economics, Dalian 116025, China, e-mail:yufudufe@gmail.com}\ \  and Nurettin Cenk Turgay \footnote{Istanbul Technical University, Faculty of Science and Letters, Department of  Mathematics, 34469 Maslak, Istanbul, Turkey
            Phone: +90(533) 227 0041, e-mail:turgayn@itu.edu.tr} \footnote{Corresponding Author}}

\maketitle

\begin{abstract}
In this paper, we study biconservative hypersurfaces in the four
dimensional Minkowski space $\mathbb E^4_1$. We give the complete
explicit classification of biconservative hypersurfaces with
diagonalizable shape operator in $\mathbb E^4_1$.



\end{abstract}

\section{Introduction}
Recall that a biharmonic map
$\phi:(M^n,g)\longrightarrow(N^m,\langle,\rangle)$ between
Riemannian manifolds is a critical point of the bienergy functional
\begin{eqnarray*}
E_2(\phi)=\frac{1}{2}\int_M|\tau(\phi)|^2v_g,
\end{eqnarray*}
 where $\tau(\phi)=
{\rm trace \nabla d\phi}$ is the tension field of $\phi$. For a
biharmonic map, the bitension field satisfies the following
associated Euler-Lagrange equation
\begin{eqnarray*}
\tau_2(\phi)=-\Delta\tau(\phi)-{\rm trace}
R^N(d\phi,\tau(\phi))d\phi=0,
\end{eqnarray*}
where $R^N$ is the curvature tensor of $N$.

If the isometric immersion $\phi$ is a biharmonic map, then $M^n$ is
called a biharmonic submanifold of $N^m$. In last years, the
research on biharmonic maps and biharmonic submanifolds is quite
active, cf. [1-5, 9-12, 16-18, 24-26]. In particular, there is a
long standing biharmonic conjecture, posed by B. Y. Chen in 1991,
that \emph{every biharmonic submanifolds in a Euclidean space is
minimal}. The conjecture is still open so far, see Chen's book
\cite{ChenKitapTotal2ndEd} for recent progress.

For an isometric immersion $\phi$, the stress-energy tensor for the
bienergy is defined as (see \cite{CMOP})
\begin{eqnarray*}
S_2(X,Y)=\frac{1}{2}|\tau(\phi)|^2\langle X,Y\rangle+\langle
d\phi,\nabla\tau(\phi)\rangle\langle X,Y\rangle\\-\langle
d\phi(X),\nabla_Y\tau(\phi)\rangle-\langle
d\phi(Y),\nabla_X\tau(\phi)\rangle,
\end{eqnarray*}
which satisfies
\begin{eqnarray}\label{eq:stress-energy}
{\rm div} S_2 =-\tau_2(\phi)^\top.
\end{eqnarray}
An immersion (or a submanifold) is called {\em biconservative} if
${\rm div} S_2=0$ (see \cite{CMOP} for details).

Note that, for an isometric immersion $\phi$, the formula
\eqref{eq:stress-energy} means that the condition ${\rm div} S_2=0$
is equivalent to the vanishing tangent part of the corresponding
bitension field, i.e., $\tau_2(\phi)^\top=0$. Hence, the notion of
biconservative submanifolds is a natural generalization of
biharmonic submanifolds.

The study of biconservative submanifolds has recently received much
attention. Caddeo et al.  classified biconservative surfaces in the
three-dimensional Riemannian space forms, \cite{CMOP}. Hasanis and
Vlachos classified biconservative hypersurfaces in the Euclidean
spaces $\mathbb E^3$ and $\mathbb E^4$ in \cite{HsanisField}, where
the authors called biconservative hypersurfaces as
\emph{H-hypersurfaces}. Chen and Munteanu \cite{chenMunteanu2013}
showed that a $\delta(2)$-ideal biconservative hypersurface in
Euclidean space $\mathbb E^n$ is either minimal or open part of a
spherical hypercylinder. By using the framework of equivariant
differential geometry, Montaldo, Oniciuc and Ratto \cite{MOR20141}
studied $SO(p+1)\times SO(q+1)$-invariant and $SO(p+1)$-invariant
biconservative hypersurfaces in Euclidean space. Most recently, the
second author obtained the complete classification of biconservative
hypersurfaces with three distinct principal curvatures in Euclidean
spaces, \cite{TurgayHHypersurfaces}.

In the case of codimension greater than one, the situation is more
difficult without any additional assumptions just as the biharmonic
case. Montaldo et al. \cite{MOR20142} studied biconservative
surfaces in Riemannian manifolds. In particular, they gave a
complete classification of biconservative surfaces with constant
mean curvature in Euclidean 4-space. Very recently, Fetcu et al.
classified biconservative surfaces with parallel mean curvature
vector field in product spaces $\mathbb S^n\times \mathbb R$ and $\mathbb H^n\times \mathbb R$ in
\cite{FOP2014}.

The notion of biconservative submanifolds was also considered in the
context of pseudo-Riemannian geometry. The first author in
\cite{FuBiconservativeE31} and \cite{Fu2} classified biconservative
surfaces in the 3-dimensional Lorentzian space forms.

In this paper, we focus on biconservative hypersurfaces in Minkowski
space $\mathbb E^4_1$. For hypersurfaces in Minkowski space, the
shape operator can be decomposed into four canonical forms, see \cite{ONeillKitap}.
We give the complete explicit classification of biconservative
hypersurfaces with diagonalizable shape operator in $\mathbb E^4_1$.
It should be remarked that, just as the case of biharmonic submanifolds,
the geometry of biconservative submanifolds in pseudo-Riemannian
space is quite different from the Riemannian case. There are more
examples of biconservative submanifolds appearing in the
classification results, see Theorem \ref{ClassThemCase12Dist} and Theorem \ref{FullClassThemCase13Dist}.

\section{Prelimineries}

Let $\mathbb E^m_t$ denote the pseudo-Euclidean $m$-space with the canonical
pseudo-Euclidean metric tensor of index $t$ given by
$$
 g=\langle\ ,\ \rangle=-\sum\limits_{i=1}^t dx_i^2+\sum\limits_{j=t+1}^m dx_j^2.
$$
We put
\begin{eqnarray}
\mathbb S^{m-1}_t(r^2)&=&\{x\in\mathbb E^m_t: \langle x, x \rangle=r^{-2}\},\notag
\\
\mathbb H^{m-1}_{t-1}(-r^2)&=&\{x\in\mathbb E^m_t: \langle x, x \rangle=-r^{-2}\}\nonumber.
\end{eqnarray}

Consider an oriented hypersurface $M$  of the  Minkowski space $\mathbb E^{n+1}_1$ with the unit normal vector field $N$ associated with the orientation. We denote Levi-Civita connections of $\mathbb E^{n+1}_1$ and $M$ by $\widetilde{\nabla}$ and $\nabla$, respectively and let $\nabla^\perp$ stand for the normal connection of $M$. Then, the Gauss and Weingarten formulas are given, respectively, by
\begin{eqnarray}
\nonumber \widetilde\nabla_X Y&=& \nabla_X Y + h(X,Y),\\
\nonumber \widetilde\nabla_X N&=& -SX
\end{eqnarray}
for all tangent vectors fields $X,\ Y$,  where $h$ and $S$ are the second fundamental form and  the shape operator of $M$,  respectively. The Gauss and Codazzi equations are given, respectively, by
\begin{eqnarray}
\label{MinkGaussEquation}\label{GaussEq} \langle R(X,Y)Z,W\rangle&=&\langle h(Y,Z),h(X,W)\rangle-
\langle h(X,Z),h(Y,W)\rangle,\\
\label{MinkCodazzi} (\bar \nabla_X h )(Y,Z)&=&(\bar \nabla_Y h )(X,Z),
\end{eqnarray}
where $R$ is the curvature tensor associated with the connection
$\nabla$ and  $\bar \nabla h$ is defined by
$$(\bar \nabla_X h)(Y,Z)=\nabla^\perp_X h(Y,Z)-h(\nabla_X Y,Z)-h(Y,\nabla_X Z).$$
$M$ is said to be biconservative if its shape operator $S$ and mean curvature $H=\mathrm{tr}S$ satisfy
\begin{BicEq}\label{BiconservativeEq}
 S(\nabla H)+\varepsilon\frac{nH}2\nabla H=0,
\end{BicEq}
where $\varepsilon=\langle N,N\rangle$, i.e.,
$$\varepsilon=\left\{\begin{array} {lcl}
-1 &\quad&\mbox{if } M\mbox{ is Riemannian}\\
1 &\quad&\mbox{if } M\mbox{ is Lorentzian}
\end{array} \right..$$

Note that the biconservative condition \eqref{BiconservativeEq}
follows directly from
 $\tau_2(\phi)^\top=0$ as we described in Introduction, see
 \cite{CMOP}.
 \begin{Remark}
The shape operator of a hypersurface with constant mean curvature satisfies  \eqref{BiconservativeEq} trivially. Therefore, throughout  this work we will assume that $\nabla H$ does not vanish on $M$.
\end{Remark}
\section{Biconservative Hypersurfaces}
Let $M$ be an oriented hypersurface with the diagonalizable shape operator $S$ in $\mathbb E^4_1$. Consider an orthonormal frame field $\{e_1,e_2,e_3\}$ of $M$ consisting of its principal directions and let $\{\theta_1,\theta_2,\theta_3\}$ be the dual base field and $k_1,k_2,k_3$ corresponding principal directions. Then, we have
$k_1+k_2+k_3=3H$.

Now, assume that $M$ is biconservative, i.e., $S$ and $H$ satisfy \eqref{BiconservativeEq} for $n=3$. Thus, we have $\nabla H$ is a principal direction with the corresponding principal curvature proportional to $H$ by a constant. Therefore, we may assume $e_1=\nabla H/|\nabla H|$ and $k_1=\frac{-3\varepsilon}2H$. Since $e_1$ is proportional to $\nabla k_1$, we have
\begin{equation}\label{BiconservativeEqERes1}
e_2(k_1)=e_3(k_1)=0, \quad e_1(k_1)\neq0.
\end{equation}
In addition, similar to biconservative hypersurfaces in Euclidean spaces, connection forms of $M$  satisfy
\begin{equation}\label{CONNFORMS}
\omega_{12}(e_1)=\omega_{12}(e_3)=\omega_{13}(e_1)=\omega_{13}(e_2)=0,
\end{equation}
and
\begin{equation}\label{CONNFORMS2}
\omega_{23}(e_1)=0, \quad \mbox{if } k_2\neq k_3
\end{equation}
(see \cite{HsanisField,TurgayHHypersurfaces}).

Let $D$ be the two-dimensional distribution given by
\begin{equation}\label{EquationDistD}
D(m)=\mathrm{span}\{e_2|_m,e_3|_m\}.
\end{equation}
\begin{Remark}Since \eqref{BiconservativeEqERes1} implies $[e_2,e_3](k_1)=0$ and $e_1$ is proportional to $\nabla k_1$, we have $\langle[e_2,e_3],e_1\rangle=0$ which gives $[e_2,e_3]_m\in D(m)$. Therefore, $D$ is involutive.\end{Remark}

First, we obtain the following lemma.
\begin{Lemma}
Let $M$ be a biconservative hypersurface in $\mathbb E^4_1$ with the diagonalizable shape operator. Then, its principal curvatures satisfy
$$e_i(k_2)=e_i(k_3)=0,\quad i=2,3.$$
\end{Lemma}

\begin{Remark}\label{REMARKPoincare}
By combining \eqref{CONNFORMS} with Cartan's first structural equation one can obtain $d\theta_1=0$, i.e., $\theta_1$ is closed. The Poincar\'e Lemma implies that it is exact, i.e., there exists a local coordinate system  $(s,\hat t,\hat u)$ on a neighborhood of $m\in M$ such that $\theta_1=ds$ from which we obtain $e_1= \frac\partial{\partial s}$. Thus, we have $k_1=k_1(s)$, $k_i=k_i(s,\hat t,\hat u),\ i=2,3$. Since $k_1'(s)\neq0$ because of \eqref{BiconservativeEqERes1}, the inverse function theorem implies $s=s(k_1)$ on a neighborhood $\mathcal N_m$ of $m$ in $M$ and we have $k_i=k_i(k_1,t,u)$. We will prove $k_i=k_i(k_1)$ on $\mathcal N_m$.
\end{Remark}
\begin{Remark}\label{RemarkBe4LemmaCase1}
Note that, since a further computation yields
$$e_ie_1^n(k_1)=0,\quad i=2,3,\ n\in\mathbb N,$$
we have $e_1^n(k_1)=f_n(k_1)$ for a smooth function $f_n$ on $\mathcal N_m$, where $e_1^n=\underbrace{e_1e_1\hdots e_1}_{n\mbox{-times}}$.
\end{Remark}

\begin{proof}
If $k_2=k_3$ proof directly follows from the Codazzi equation \eqref{MinkCodazzi} for $X=e_2$, $Y=e_3$, $Z=e_2$ and $X=e_3$, $Y=e_2$, $Z=e_3$. Thus, will assume that $k_2-k_3$ does not vanish on $M$. We have two cases subject to being Riemannian or Lorentzian of $M$.

\textit{Case I.} $M$ is Riemannian. In this case, we have $\varepsilon_1=-1$ and \eqref{BiconservativeEq} gives
\begin{equation}\label{Case1RiemannianEq1}
k_1=k_2+k_3.
\end{equation}
The Codazzi equation \eqref{MinkCodazzi} for $X=e_1,Y=Z=e_i$ implies
$$e_1(k_i)=\omega_i(k_1-k_i).$$
In addition, by combining \eqref{CONNFORMS} with the Gauss equation $R(e_i,e_1,e_1,e_i)=k_1k_i$ we have
$$e_1(\omega_i)=-\omega_i^2-k_1k_i,$$
where we put $\omega_i=\omega_{1i}(e_i).$ We apply $e_1$ 3 times to \eqref{Case1RiemannianEq1} to obtain
\begin{subequations}\label{Case1RiemannianEq1ALL}
\begin{eqnarray}
\label{Case1RiemannianEq1b}
e_1(k_1)&=&\left(k_1-k_2\right) \omega_2+\left(k_1-k_3\right) \omega_3,\\
\label{Case1RiemannianEq1c}
e_1^2(k_1)+2k_1 k_2 k_3&=&\left(\omega_2+\omega_3\right) e_1(k_1)-2\left(k_1-k_2\right) \omega_2^2-2\left(k_1-k_3\right) \omega_3^2
\end{eqnarray}
and
\begin{align}
\begin{split}
\label{Case1RiemannianEq1d}
\left(2 k_2 k_3+k_1^2\right) e_1(k_1)+e_1^3(k_1)=&6(k_1-k_2)\omega_2^3+6(k_1-k_3)\omega_3^3-3e_1(k_1)(\omega_2^2\omega_3^2)\\
&+\left(e_1^2(k_1)+2 k_1 \left(k_1-k_2\right) \left(2 k_2-k_3\right)\right)\omega_2\\
&+\left(e_1^2(k_1)+2 k_1 \left(k_3-k_1\right) \left(k_2-2 k_3\right)\right)\omega_3.
\end{split}
\end{align}
\end{subequations}
Note that from \eqref{Case1RiemannianEq1} and \eqref{Case1RiemannianEq1b} we get
\begin{equation}\label{Case1RiemannianEq2}
\omega_3= \frac{e_1(k_1)-(k_1 -k_2) \omega_2}{k_2}.
\end{equation}

Next, we use \eqref{Case1RiemannianEq1} and \eqref{Case1RiemannianEq2} on \eqref{Case1RiemannianEq1c} and \eqref{Case1RiemannianEq1d} to get
\begin{eqnarray}
\label{Case1RiemannianEq4b}
A_1\omega_2^3+A_2\omega_2^2+A_3\omega_2+A_4&=&0,\\
\label{Case1RiemannianEq4c}
B_1\omega_2^2+B_2\omega_2+B_3&=&0,
\end{eqnarray}
where $A_j$ and $B_j$ are the functions given by
\begin{align}\nonumber
\begin{split}
A_1=&6 k_1 \left(k_1-2 k_2\right) \left(k_1-k_2\right),\\
A_2=&-3 \left(5 k_1^2-10 k_2 k_1+4 k_2^2\right) e_1(k_1),\\
A_3=&6 \left(k_1-k_2\right) \left(2 e_1(k_1)^2+k_1 \left(k_1-2 k_2\right) k_2^2\right)-\left(k_1-2 k_2\right) k_2 e_1^2(k_1),\\
A_4=&-e_1(k_1) \left(3 e_1(k_1)^2+k_2 \left(e_1^2(k_1)+2 k_2^3-8 k_1 k_2^2+3 k_1^2 k_2\right)\right)+k_2^2 e_1^3(k_1),\\
B_1=&2 k_1 \left(k_2-k_1\right),\\
B_2=&\left(3 k_1-2 k_2\right) e_1(k_1),\\
B_3=&-e_1(k_1)^2-k_2 \left(e_1^2(k_1)+2 k_1 k_2 (k_1-k_2)\right).
\end{split}
\end{align}
Finally, we eliminate $\omega_2$ from \eqref{Case1RiemannianEq4b} and \eqref{Case1RiemannianEq4c} to get
\begin{align}
\begin{split}4 A_1 B_1 \left(2 A_3 B_1 \left(\delta^2-B_2^4\right)+B_2 \left(4 A_4 B_1^2 \left(B_2^2+3 \delta\right)+A_2 \left(\delta-B_2^2\right)^2\right)\right)\\
+A_1^2 \left(\delta-B_2^2\right)^3-4 B_1^2 \left(4 B_1^2 \left(A_3^2 \left(B_2^2-\delta\right)-4 A_4 A_3 B_1 B_2+4 A_4^2 B_1^2\right)\right.\\
\left.+4 A_2 B_1 \left(A_3 B_2 \left(\delta-B_2^2\right)+2 A_4 B_1 \left(B_2^2+\delta\right)\right)+A_2^2 \left(\delta-B_2^2\right)^2\right),
\end{split}
\end{align}
where $\delta=B_2^2-4B_1B_3$. Next, we put $A_i,B_i$ into the equation above to obtain a 14th degree polinomial
$$\sum\limits_{j=0}^{14} P_j(k_1,e_1(k_1),e_1^2(k_1),e_1^3(k_1))k_2^j=0$$
with the starting term $P_{14}=-16384  k_1^2e_1(k_1).$ However, Remark \ref{RemarkBe4LemmaCase1} implies
$$P_j\left(k_1,e_1(k_1),e_1^2(k_1),e_1^3(k_1)\right)=Q_j(k_1)$$ for a function $Q_j$. Therefore, we have
$$\sum\limits_{j=0}^{14} Q_j(k_1)k_2^j=0.$$
Thus, $k_2$ is depending on only $k_1$. Moreover, \eqref{Case1RiemannianEq1} implies that $k_3$ is also depending on only $k_1$.

\textit{Case II.} $M$ is Lorentzian. In this case, we have $\varepsilon_1=-1$ and \eqref{BiconservativeEq} gives
\begin{equation}\label{Case1LorentzianEq1}
-3k_1=k_2+k_3.
\end{equation}
By a similar way, we obtain
$$\sum\limits_{j=0}^{14} \tilde Q_j(k_1)k_2^j=0$$
for some functions $\tilde Q_j$.
Thus, $k_2$ and $k_3$ are depending on only $k_1$. Hence the proof is completed.
\end{proof}

Now, we have $e_i(k_j)=0$ which implies $e_ie_1(k_j),\ i,j=2,3$. Therefore, the Codazzi equation  $e_1(k_i)=\omega_{1i}(e_i)(k_1-k_i)$ implies
\begin{equation}\label{eiejomega}
e_i(\omega_{1j}(e_j))=0.
\end{equation}
 Hence, $\omega_{12}(e_2), \omega_{13}(e_3)$ are constant on  any integral submanifold $\hat M$ of the distribution $D$ given by \eqref{EquationDistD}. Let $c_i,d_i$ are the constants given by
\begin{equation}\label{cdConstants}
d_1=\left.k_2\right|_{\hat M},d_2=\left.k_3\right|_{\hat M},c_1=\left.\omega_{12}(e_2)\right|_{\hat M},c_2=\left.\omega_{13}(e_3)\right|_{\hat M}
\end{equation}
and consider  the local orthonormal frame field $\{f_1,f_2;f_3,f_4\}$ consisting of restriction of vector fields $e_2,e_3,e_1,N$ to $\hat M$, respectively. Then, we have
\begin{Lemma}\label{Case1LemmaShapeOperators}
$f_3$ and $f_4$ are parallel and the matrix representations of the shape operators $\hat A_{f_3}$ and $\hat A_{f_4}$  are
$$\hat A_{f_3}=\mathrm{diag}(c_1,c_2),\ \hat  A_{f_4}=\mathrm{diag}(d_1,d_2),$$
where $c_i,d_i$ are the constants given by \eqref{cdConstants}.
\end{Lemma}
\noindent Moreover, we have
\begin{Corol}
$\hat M$ has parallel mean curvature vector in $\mathbb E^4_1$.
\end{Corol}
In addition, if $M$ has three distinct principal curvatures, then by combining $e_i(k_2)=e_i(k_3)=0$ with the Codazzi equation \eqref{MinkCodazzi} and taking into account \eqref{CONNFORMS2}, one can see that the connection form $\omega_{23}$ vanishes identically. Therefore, we have
\begin{Lemma}\label{LemmaLeviCivitaThreeDistCurve}
If $M$ is a biconservative hypersurface in $\mathbb E^4_1$ with three real, distinct principal curvatures, then the Levi-Civita connection $\hat\nabla$ of $\hat M$  satisfies $\hat\nabla_{f_i}f_j=0,\ i,j=1,2$. Consequently,  $\hat M$ is flat.
\end{Lemma}

We have the following proposition (see also [20, Lemma 4.2]).
\begin{Prop}\label{Prop1CoordSystem}
Let $M$ be  a biconservative hypersurface with diagonalizable shape operator in $\mathbb E^4_1$. Then, there exists a local coordinate system $(s,t,u)$ such that
\begin{equation}\label{Case1CloseToClasssificationEq00}
e_1=\frac\partial{\partial s},\quad e_2=\frac1{E_1}\frac\partial{\partial t},\quad e_3=\frac1{E_2}\frac\partial{\partial u}.
\end{equation}
\end{Prop}
\begin{proof}
 Let $D^\perp$ be the distribution given by $D^\perp(m)=\mathrm{span}\{e_3|_m\}$. Since $D^\perp$ and $D$ are involutive and $D(m)\oplus D^\perp(m)=T_mM$, by using [19, Lemma in page 182], we see that there is a local coordiane system $(\hat s,t,u)$ on $M$ such that $e_1$ is proportional to $\partial_{\hat s}$ and
$e_2=\frac1{E_1}\frac\partial{\partial t},\ e_3=\frac1{E_2}\frac\partial{\partial u}$.  Let $(s,\hat t, \hat u)$ be the coordinate system given in Remark \ref{REMARKPoincare}. Then, the local coordinate system $(s,t,u)$ satisfies the condition given in the proposition.
\end{proof}

Next, we obtain a local parametrization of biconservative hypersurfaces.
\begin{Prop}\label{PropCase1CloseToClasssificationEq000}
Let $M$ be  a biconservative hypersurface with diagonalizable shape operator in $\mathbb E^4_1$. If $M$ has two distinct principal curvature, then it has a local parametrization
\begin{equation}\label{Case1CloseToClasssificationEq001}
x(s,t,u)=\phi(s)\Theta(t,u)+\Gamma(s)
\end{equation}
for some vector valued functions $\Theta,\Gamma$ and a function $\phi.$
 On the other hand, if $M$ has three distinct principal curvature, then $M$ has a local  parametrization
\begin{equation}\label{Case1CloseToClasssificationEq000}
x(s,t,u)=\phi_1(s)\Theta_1(t)+\phi_2(s)\Theta_2(u)+\Gamma(s)
\end{equation}
for some vector valued functions $\Theta_1,\Theta_2,\Gamma$ and functions $\phi_1,\phi_2.$
\end{Prop}
\begin{proof}
Because of \eqref{eiejomega}, we have $\omega_{12}(e_2)=\alpha(s)$, $\omega_{13}(e_3)=\beta(s)$ . Therefore, \eqref{CONNFORMS} implies
\begin{equation}\label{Case1CloseToClasssificationEq1}
\widetilde\nabla_{e_2}e_1=\alpha(s)e_2,\widetilde\nabla_{e_3}e_1=\beta(s)e_3.
\end{equation}
Let $x$ be the position vector of $M$ and $(s,t,u)$ the coordinate system given in Proposition \ref{Prop1CoordSystem}. If $M$ has two distinct principal curvatures, then we have $k_2=k_3$ which implies $\alpha=\beta$. Therefore, from \eqref{Case1CloseToClasssificationEq00} and \eqref{Case1CloseToClasssificationEq1} we have $x_{st}=\alpha(s)x_t,\ x_{su}=\alpha(s)x_u$. By integrating these equations, we obtain \eqref{Case1CloseToClasssificationEq001}.

Now, suppose  that $M$ has three distinct principal curvatures. Then, from \eqref{Case1CloseToClasssificationEq00} and \eqref{Case1CloseToClasssificationEq1} we have
$$x_{st}=\alpha(s)x_t,\quad x_{su}=\beta(s)x_u.$$
By integrating these equations, we obtain \eqref{Case1CloseToClasssificationEq000}.
\end{proof}

\begin{Lemma}\label{Case1IntSub1}
Let $M$ be a biconservative hypersurface in $\mathbb E^4_1$ with the diagonalizable shape operator and $\hat M$ the integral submanifold of the distribution $D$ given
 by \eqref{EquationDistD} passing through $m\in M$. Then, if $M$ has three distinct principal curvatures, then $\hat M$ is congruent to one of the surfaces given by
\begin{enumerate}
\item[(i)] A Riemannian surface lying on a Euclidean hyperplane of  $\mathbb E^4_1$ given by $$y(t,u)=(1,t,B \cos u,B\sin u);$$
\item[(ii)] A Riemannian surface lying on a Lorentzian hyperplane of  $\mathbb E^4_1$  given by $$y(t,u)=(A\mathrm{cosh}t,A\mathrm{sinh}t,u,1);$$
\item[(iii)] A Lorentzian surface lying on a Lorentzian hyperplane of  $\mathbb E^4_1$  given by $$y(t,u)=(t,B \cos u,B\sin u,1);$$
\item[(iv)] A Lorentzian surface lying on a Lorentzian hyperplane of  $\mathbb E^4_1$  given by $$y(t,u)=(A\mathrm{sinh}t,A\mathrm{cosh}t,u,1);$$
\item[(v)] A Riemannian torus $\mathbb H^1(-A^2)\times \mathbb S^1(B^2)$ given by $$y(t,u)=(A\mathrm{cosh}t,A\mathrm{sinh}t,B \cos u,B\sin u);$$
\item[(vi)]  A Lorentzian torus $\mathbb S^1_1(A^2)\times \mathbb  S^1(B^2)$ given by $$y(t,u)=(A\mathrm{sinh}t,A\mathrm{cosh}t,B \cos u,B\sin u);$$
\item[(vii)] A Riemannian surface lying on a degenerated hyperplane of  $\mathbb E^4_1$  given by $$y(t,u)=(At^2+Bu^2,t,u,At^2+Bu^2).$$
\end{enumerate}
On the other hand,  if $M$ has two distinct principal curvatures, then $\hat M$ is congruent to one of the surfaces given by
\begin{enumerate}
\item[(viii)] A sphere $\mathbb S^2(r^2)\subset\mathbb E^3\subset\mathbb E^4_1$;
\item[(ix)] de Sitter space $\mathbb S^2_1(r^2)\subset\mathbb E^3_1\subset\mathbb E^4_1$;
\item[(x)] anti-de Sitter space $\mathbb H^2(-r^2)\subset\mathbb E^3_1\subset\mathbb E^4_1$;
\item[(xi)] The flat marginally trapped surface $y(t,u)=\left(A(t^2+u^2),t,u,A(t^2+u^2)\right)$.
\end{enumerate}
\end{Lemma}

\begin{proof}
Let $\hat M$ be an integral submanifold of the distribution $D$ and $y$ the position vector of $\hat M$. Consider the  local orthonormal frame field $\{f_1,f_2;f_3,f_4\}$ on $M$ given before the Lemma \ref{Case1LemmaShapeOperators}. We study the cases $k_2\neq k_3$ and $k_2= k_3$ separately.

\textit{Case 1.} First, assume that $M$ has three distinct principal curvatures, i.e., $k_2\neq k_3$.  Without loss of generality, we may assume $\varepsilon_2=\varepsilon_4=1$ which gives $\varepsilon_1\varepsilon_3=-1$. Then, we have
\begin{align}
\begin{split}\label{Case1IntSub1Eq0}
\widetilde\nabla_{f_1}f_1=-c_1f_3+\varepsilon_1d_1f_4, \quad \widetilde\nabla_{f_1}f_2=\widetilde\nabla_{f_2}f_1=0, \quad \widetilde\nabla_{f_2}f_2=\varepsilon_3c_2f_3+d_2f_4,\\
\widetilde\nabla_{f_1}f_3=-c_1f_1, \quad \widetilde\nabla_{f_1}f_4=-d_1f_1, \quad \widetilde\nabla_{f_2}f_3=-c_2f_2, \quad\widetilde\nabla_{f_2}f_4=-d_2f_2
\end{split}
\end{align}
because of Lemma \ref{LemmaLeviCivitaThreeDistCurve}. Since $\hat M$ is flat and $\hat\nabla_{f_i}f_j=0,$ there exists a local coordinate system $(t,u)$ such that $g=\varepsilon_1 dt^2+du^2$, $f_1=\partial_t$ and $f_2=\partial_u$, where $\hat\nabla$ is the Levi-Civita connection of $\hat M$. Thus, $\widetilde\nabla_{f_2}f_1=0$ implies
\begin{equation}\label{Case1IntSub1Eq1}
y(t,u)=\alpha(t)+\beta(u)
\end{equation}
for some smooth vector valued functions $\alpha,\beta$.
From \eqref{Case1IntSub1Eq0} and \eqref{Case1IntSub1Eq1}, we obtain
\begin{subequations}\label{Case1IntSub1Eq2ALL}
\begin{eqnarray}
\label{Case1IntSub1Eq2a} \alpha'''&=&(c_1^2-\varepsilon_1d_1^2)\alpha',\\
\label{Case1IntSub1Eq2b} \beta'''&=&-(\varepsilon_3c_2^2+d_2^2)\beta'.
\end{eqnarray}
\end{subequations}
Moreover, since $\hat M$ is flat, we have
\begin{equation}\label{Case1IntSub1Eq3}
\varepsilon_1d_1d_2-c_1c_2=0.
\end{equation}

\textit{Case 1a.} $\varepsilon_1=-1$, i.e., $\hat M$ is Lorentzian. In this case, \eqref{Case1IntSub1Eq2ALL} implies
\begin{subequations}\label{Case1IntSub1Eq4ALL}
\begin{eqnarray}
\label{Case1IntSub1Eq4a} \alpha'''&=&\nu^2\alpha',\\
\label{Case1IntSub1Eq4b} \beta'''&=&-\mu^2\beta'
\end{eqnarray}
\end{subequations}
for some positive constants $\nu,\mu$. Since $\nu=\mu=0$ implies that $\hat M$ is a plane which yields a contradiction, we have $\nu^2+\mu^2\neq0$. Thus, if $\nu=0$, then $\mu\neq0$. In this case solving \eqref{Case1IntSub1Eq4ALL} yields that
$$y(t,u)=t^2\eta_1+t\eta_2+\cos(\mu u)\eta_3+\sin(\mu u) \eta_4$$
for some  constant vectors $\eta_1,\eta_2,\eta_3,\eta_4$. By considering $g=-dt^2+du^2$, we obtain the case (iii) of the lemma. Similarly, the other possible subcases  $\mu=0$, $\nu\neq0$ and  $\mu\nu\neq0$ give the case (iv) and the case (vi), respectively.

\textit{Case 1b.} $\varepsilon_1=1$, i.e., $\hat M$ is Riemannian. In this case,  \eqref{Case1IntSub1Eq2ALL} implies
\begin{subequations}\label{Case1IntSub1Eq6ALL}
\begin{eqnarray}
\label{Case1IntSub1Eq6a} \alpha'''&=&(c_1^2-d_1^2)\alpha',\\
\label{Case1IntSub1Eq6b} \beta'''&=&(c_2^2-d_2^2)\beta'.
\end{eqnarray}
\end{subequations}
By taking into account \eqref{Case1IntSub1Eq3}, we see that, without loss of generality, we have four cases.
\begin{align}\nonumber
\begin{split}
 c_1^2-d_1^2=\nu^2,\quad& c_2^2-d_2^2=-\mu^2;\\
 c_1=d_1\neq0,\quad& c_2^2-d_2^2=0;\\
 c_1=d_1=0,\quad& c_2^2-d_2^2=\nu^2;\\
 c_1=d_1=0,\quad& c_2^2-d_2^2=-\mu^2.
\end{split}
\end{align}
By integrating \eqref{Case1IntSub1Eq6ALL} for each cases separately, we see that  $\hat M$ is congruent to one of the following surfaces.
\begin{subequations}\label{Case1IntSub1Eq7ALL}
\begin{eqnarray}
\label{Case1IntSub1Eq7a} y(t,u)&=&\mathrm{cosh}(\nu t)\eta_1+\mathrm{sinh}(\nu t) \eta_2+\cos(\mu u)\eta_3+\sin(\mu u) \eta_4\\
\label{Case1IntSub1Eq7b} y(t,u)&=&t^2\eta_1+t(\nu t) \eta_2+u^2\eta_3+ u \eta_4,\\
\label{Case1IntSub1Eq7c} y(t,u)&=&\mathrm{cosh}(\nu t)\eta_1+\mathrm{sinh}(\nu t) \eta_2+u^2\eta_3+u \eta_4\\
\label{Case1IntSub1Eq7d} y(t,u)&=&t^2\eta_1+t\eta_2+\cos(\mu u)\eta_3+\sin(\mu u) \eta_4
\end{eqnarray}
\end{subequations}
for some constant vectors $\eta_1,\eta_2,\eta_3,\eta_4$. By a direct computation using  $g=dt^2+du^2$, we obtain the case (v), (vii), (ii) and (i) of the lemma, respectively.

\textit{Case 2.} Next, we assume that $M$ is a biconservative hypersurface with two distinct principal curvatures. Then, the shape operators of $\hat M$ becomes $A_3=c_1I,\ A_{f_4}=d_1I$ by the Lemma \ref{Case1LemmaShapeOperators}. Thus, $\hat M$ lies on a hyperplane $\Pi$ of $M$ whose normal is the constant vector $\eta=\varepsilon_3d_1e_3-\varepsilon_4 c_1 e_4$.

 If $\Pi$ is non-degenerated, then $\hat M$ is isoparametric. Thus, we have the case (viii) or cases (ix), (x) subsect to being Euclidean or non-Euclidean of $\Pi$, respectively.

Now, suppose that $\Pi$ is degenerated, i.e., $\eta$ is light-like. Then, we  have $c_1=d_1$. In addition, up to congruency, we may assume $\Pi=\{(A,B,C,A)|A,B,C\in\mathbb R\}$. Thus, $M$ has a parametrization $(f(t,u),t,u,f(t,u))$. Since $A_3=A_4=c_1I$, we have the case (xi) of the lemma.
\end{proof}

\subsection{Biconservative hypersurfaces with two principal curvatures}
In this section, we would like to deal with the biconservative hypersurfaces with two distinct principal curvatures.
\begin{theorem}\label{ClassThemCase12Dist}
Let $M$ be a hypersurface in $\mathbb E^4_1$ with diagonalizable shape operator and two distinct principal curvatures. If $M$ is biconservative, then it is congruent to one of hypersurfaces
\begin{subequations}\label{Case12DistThmxALL}
\begin{eqnarray}
\label{Case12DistThmx1} x_1(s,t,u)&=&(f_1(s),s \cos t\sin u,s\sin t\sin u,s \cos u),\\
\label{Case12DistThmx2} x_2(s,t,u)&=&(s\mathrm{sinh} u\sin t,s\mathrm{cosh}u\sin t,s\cos t,f_2(s)),\\
\label{Case12DistThmx3} x_3(s,t,u)&=&(s\mathrm{cosh}t,s\mathrm{sinh}t\sin u,\mathrm{sinh}t\cos u,f_3(s)),\\
\label{Case12DistThmx4} x_4(s,t,u)&=&\left(\frac 12  s( t^2+ u^2)+s+f_4(s), s t,s u,\frac 12  s( t^2+ u^2)+f_4(s)\right)
\end{eqnarray}
\end{subequations}
for some smooth functions $f_1,f_2,f_3,f_4$.
\end{theorem}
\begin{proof}
Let $M$ be a biconservative hypersurface in $\mathbb E^4_1$ with the parametrization given by \eqref{Case1CloseToClasssificationEq001} for some vector valued functions $\Theta,\Gamma$ and a function $\phi$. Now, consider the slice $\hat M$ of $M$ given by $s=s_0$ passing through $m=x(s_0,t_0,u_0)$. Obviously, it is an integral submanifold of the distribution $D$ given by \eqref{EquationDistD}. Then, $\hat M$ is one of four surfaces given in Case (viii)-(xi) of Lemma \ref{Case1IntSub1}.

First, assume that $\hat M$ is a sphere. In this case, up to isometries of $\mathbb E^4_1$, we may assume the position vector of $\hat M$ is $y(t,u)=x(s_0,t,u)=(1, A \cos t\sin u,A \sin t\sin u, A\cos u)$. Then, \eqref{Case1CloseToClasssificationEq001} implies
$$c_1\Theta(t,u)+c_2=(1, \cos t\sin u,\sin t\sin u, \cos u)$$
for a constant $c_1$ and constant vector $c_2$. By solving $\Theta$ from the above equation and using  \eqref{Case1CloseToClasssificationEq001}, we obtain $M$ is the hypersurface given by \eqref{Case12DistThmx1}.

Analogously, if $\hat M=\mathbb S^2_1(r^2)$ or $\hat M=\mathbb H^2_1(-r^2)$, we obtain $M$ is the hypersurface given by \eqref{Case12DistThmx2} or \eqref{Case12DistThmx3}, respectively.

On the other hand, if $\hat M$ is congruent to the flat marginally trapped surface given in the Case (xi)  of Lemma \ref{Case1IntSub1}, then, up to isometries, we may assume
$$x(s_0,t,u)=y(t,u)=\left(A(t^2+u^2),t,u,A(t^2+u^2)\right).$$
By combining this equation and \eqref{Case1CloseToClasssificationEq001} we obtain
$$c_1\Theta(t,u)+c_2=\left(A(t^2+u^2),t,u,A(t^2+u^2)\right),$$
where $c_1=\phi(s_0)$ and $c_2=\Gamma(s_0)$. Therefore, we may assume
$$\Theta(t,u)=\left(A'(t^2+u^2)+C_1,t+C_2,u+C_3,A'(t^2+u^2)+C_4\right)$$
for some constant $A',C_i$. Next, we put this equation into \eqref{Case1CloseToClasssificationEq001} to get
$$x(s,t,u)=\left(\bar\phi(s)(t^2+u^2),t,u,\bar\phi(s)(t^2+u^2)\right)+\bar\Gamma(s)$$
for a smooth function $\bar\phi_1$ and smooth vector valued function $\bar\Gamma.$ By taking into account that the vector fields $\partial_s,\partial_t,\partial_u$ are orthonormal and re-defining the coordinate $s$ properly,  we obtain $\bar\phi(s)=\frac 12s$ and $\bar\Gamma(s)=(s+f_4(s),0,0,f_4(s))$. Therefore, we obtain the surface given by \eqref{Case12DistThmx4}. Hence, the proof is completed.
 \end{proof}
By the following proposition, we  would like to prove the existence of biconservative hypersurfaces with two distinct curvatures.
\begin{Prop}
Let $M$ be the hypersurface given by \eqref{Case12DistThmx1} in $\mathbb E^4_1$. Then, $M$ is biconservative if and only if either $M$ is Riemannian and
\begin{equation}\label{Propx1Ri}
f_1=\int\limits_{s_0}^s\frac{c_1 \xi^2}{\sqrt{c_1^2 \xi^4-1}}d\xi
\end{equation}
or it is Lorentzian and
\begin{equation}\label{Propx1Lo}
f_1=\int\limits_{s_0}^s\frac{c_1}{\sqrt{c_1^2-\xi^{4/3}}}d\xi.
\end{equation}
 \end{Prop}
\begin{proof}
By a direct computation one can obtain that the principal directions of $M$ are $e_1=\frac1{\sqrt{\varepsilon_1(1-f_1'^2)}}\partial_s$, $e_2=\frac 1s\partial_t$, $e_3=\frac 1s\partial_u$ with the corresponding principal curvatures
$$k_1=-\frac{\varepsilon_1f_1''}{\sqrt {\varepsilon_1(1-f_1'^2)^3}},\quad k_2=k_3=-\frac{f_1'}{s\sqrt{\varepsilon_1(1-f_1'^2)}},$$
where $\varepsilon_1=\langle e_1,e_1\rangle.$ Let $M$ be a biconservative hypersurface, i.e., \eqref{BiconservativeEq} is satisfied.

First, assume that $M$ is Riemannian, i.e., $\varepsilon=1$. Then, from \eqref{BiconservativeEq} we have $k_1=2k_2$ which implies
$$\frac{f_1''}{f_1'(1-f_1'^2)}=\frac 2s$$
whose general solution is \eqref{Propx1Ri}.

Next, assume that $M$ is Lorentzian, i.e., $\varepsilon=-1$. Then,  \eqref{BiconservativeEq} implies $-3k_1=2k_2$ from which we have
$$\frac{-3f_1''}{f_1'(1-f_1'^2)}=\frac 2s.$$
By solving this equation, we obtain \eqref{Propx1Lo}.

Hence, the proof of necessary condition is completed. The converse follows from a direct computation.
\end{proof}

\subsection{Biconservative hypersurfaces with three principal curvatures}
In this subsection we obtain the classification of biconservative hypersurfaces with three principal curvatures. First, we want to present an example by the following proposition.

\begin{Prop}\label{ExampleTheoremCase7}
Let $M$ be a hypersurface in $\mathbb E^4_1$ given by
\begin{align}\label{TheoremCase7}
\begin{split}
x(s,t,u)=&\left(\frac 12  s( t^2+ u^2)+a u^2+s+\phi(s), s t,( s+2a)  u,\right.\\
&\left.\frac 12  s( t^2+ u^2)+a u^2+\phi(s)\right), \quad a\neq0.
\end{split}
\end{align}
Then, $M$ is  biconservative if and only if  either $M$ is Riemannian and
\begin{equation}\label{ClassThemCase13DistEqdaadasfdsdx8}
\phi(s)=c_1 \left(\ln (s+2 a)- \ln s-\frac{a}{s}-\frac{a}{s+2 a}\right)-\frac{s}{2}
\end{equation}
or it is Lorentzian and
\begin{equation}\label{ClassThemCase13DistEqdaadadx8}
\phi(s)=c_1\int\limits_{s_0}^s\left(\xi(\xi+2a)\right)^{2/3}d\xi-\frac s2,
\end{equation}
where $c_1\neq0$ and $s_0$ are some constants.
\end{Prop}

\begin{proof}
By a direct computation, one can obtain that the principal directions of $M$ are
$$e_1=\frac{1}{\sqrt{\varepsilon_1(-2\phi'-1)}}\partial_s,\quad e_2=\frac 1s\partial_t,\quad e_3=\frac 1{s+2a}\partial_u$$
and the unit normal vector of $M$ is
$$N=\frac{1}{\sqrt{\varepsilon_1(-2\phi'-1)}}\left(\frac{t^2+u^2}2-\phi',t,u,\frac{t^2+u^2}2-\phi'-1\right).$$
By a simple calculation, we have
\begin{align}\label{ClassThemCase13DistEqx8}
\begin{split}
\widetilde\nabla_{e_1}e_1&=\zeta e_1+\frac{\phi''}{\varepsilon_1(-2\phi'-1)}(1,0,0,1),\\
\widetilde\nabla_{e_2}e_2&=\frac{1}{s}(1,0,0,1),\\
\widetilde\nabla_{e_3}e_3&=\frac{1}{s+2a}(1,0,0,1)
\end{split}
\end{align}
for a smooth function $\zeta$.

\textbf{Riemannian Case.} If $M$ is Riemannian, then we have $\varepsilon_1=1$. In this case, $M$ is biconservative if and only if equation $k_1=k_2+k_3$ is satisfied. Thus, we have
$$-\frac{\phi''}{(2\phi'+1)}=\frac{1}{s}+\frac{1}{s+2a}$$
whose general solution is
$$\phi(s)=c_1 \left(\ln (s+2 a)- \ln s-\frac{a}{s}-\frac{a}{s+2 a}\right)-\frac{s}{2}+c_2,$$
where $c_1,c_2$  are constants. Note that, up to congruency, we may assume $c_2=0$. Thus, we have \eqref{ClassThemCase13DistEqdaadasfdsdx8}.

\textbf{Lorentzian Case.} If $M$ is Lorentzian, then we have $\varepsilon_1=-1$ and $M$ is biconservative if and only if  $-3k_1=k_2+k_3$ which implies
$$\frac{3\phi''}{2\phi'+1}=\frac{1}{s}+\frac{1}{s+2a}$$
whose general solution is the function given in \eqref{ClassThemCase13DistEqdaadadx8}.
\end{proof}

Next, we obtain the following classification theorem of biconservative hypersurfaces with three distinct principal curvatures.
\begin{theorem}\label{FullClassThemCase13Dist}
Let $M$ be a hypersurface in $\mathbb E^4_1$ with diagonalizable shape operator and three distinct principal curvatures. Then $M$ is biconservative if and only if it is congruent to one of hypersurfaces
\begin{enumerate}
\item[(i)] A generalized cylinder $M^2_0\times \mathbb E^1_1$ where $M$ is a biconservative surface in $\mathbb E^3$;
\item[(ii)] A generalized cylinder $M^2_0\times \mathbb E^1$ where $M$ is a biconservative Riemannian surface in $\mathbb E^3_1$;
\item[(iii)]A generalized cylinder $M^2_1\times \mathbb E^1$, where $M$ is a biconservative Lorentzian surface in $\mathbb E^3_1$;
\item[(iv)] A Rimennian surface given by
\begin{equation}\label{TheoremCase4}
x(s,t,u)=\left(s\mathrm{cosh}t,s\mathrm{sinh}t,f_1(s) \cos u,f_1(s)\sin u\right)
\end{equation}
 for a function $f_1$ satisfying
\begin{equation}\label{TheoremCase4ODE}
\frac{f_1''}{f_1'^2-1}=\frac{f_1f_1'+s}{sf_1};
\end{equation}
\item[(v)] A Lorentzian surface with the parametrization given in \eqref{TheoremCase4} for a function $f_1$ satisfying
\begin{equation}\label{TheoremCase4ODE2}
\frac{-3f_1''}{f_1'^2-1}=\frac{f_1f_1'+s}{sf_1};
\end{equation}
\item[(vi)] A Rimennian surface  given by
\begin{equation}\label{TheoremCase6}
x(s,t,u)=\left(s\mathrm{sinh}t,s\mathrm{cosh}t,f_2(s) \cos u,f_2(s)\sin u\right)
\end{equation}
 for a function $f_2$ satisfying
$$\frac{f_2''}{f_2'^2+1}=\frac{f_2f_2'+s}{sf_2};$$
\item[(vii)] A  surface given in Proposition \ref{ExampleTheoremCase7}.
\end{enumerate}
\end{theorem}

\begin{proof}
Consider a biconservative hypersurface $M$ with three distinct curvatures and assume that the functions $k_1-k_2$, $k_1-k_3$ and $k_2-k_3$ are non-vanishing in $M$. Let $\hat M$ be the  integral submanifold of the distribution $D$ given by \eqref{EquationDistD} passing through $m=x(s_0,t_0,u_0)\in M$, where $x$ is the parametrization of $M$ given in  \eqref{Case1CloseToClasssificationEq000}. Then, $\hat M$ is congruent to one of the surfaces given in case (i)-(vii) of {Lemma} \ref{Case1IntSub1}.

\textbf{Case 1.} $\hat M$ is congruent to one of the surfaces given in the case (i)-(iv) of {Lemma} \ref{Case1IntSub1}. In this case, by a direct computation one can see that one of the principal curvatures  of $M$ vanishes identically. Therefore, we have case (i)-(iii) of the theorem.


\textbf{Case 2.} Let $\hat M$ be congruent to  the surfaces given in case (v) of {Lemma} \ref{Case1IntSub1}. Then, we may assume
$x(s_0,t,u)=y(t,u)=(A\mathrm{cosh}t,A\mathrm{sinh}t,B \cos u,B\sin u).$
By combining this equation and \eqref{Case1CloseToClasssificationEq000}, we have
\begin{equation}\label{Tempor1}
c_1\Theta_1(t)+c_2(s)\Theta_2(u)+c_3=(A\mathrm{cosh}t,A\mathrm{sinh}t,B \cos u,B\sin u),
\end{equation}
where $c_1=\phi_1(s_0),c_2=\phi_2(s_0)$ and $c_3=\Gamma(s_0)$. By redefining $\phi_1,\phi_2,\Gamma$ suitable, from \eqref{Tempor1} we obtain
\begin{align}
\begin{split}
\Theta_1(t)=&(\mathrm{cosh}t,\mathrm{sinh}t,0,0),\\
\Theta_2(u)=&(0,0,\cos u,\sin u).
\end{split}
\end{align}
Therefore, \eqref{Case1CloseToClasssificationEq000} implies
$$x(s,t,u)=\left(\phi_1(s)\mathrm{cosh}t,\phi_1(s)\mathrm{sinh}t,\phi_2(s) \cos u,\phi_2(s)\sin u\right)+\Gamma(s).$$
By a further computation considering that $\partial_s,\partial_t,\partial_u$, we see that $\Gamma$ is a constant vector which can be assumed to be zero up to a suitable translation. By using the  inverse function theorem, we assume $\phi_1(s)=s$ and $\phi_2(s)=f_1(s)$ for a smooth function $f$. Hence, we have \eqref{TheoremCase4}.

A direct computation shows that the principal curvatures of $M$ are
$\displaystyle k_1=\frac{\varepsilon f_1''}{\sqrt{\varepsilon(f_1'^2-1)^3}}$, $\displaystyle k_2=\frac{ f_1'}{s\sqrt{\varepsilon(f_1'^2-1)}}$ and $\displaystyle k_3=\frac{1}{\sqrt{\varepsilon(f_1'^2-1)}},$ where $\varepsilon=1$ or $\varepsilon=-1$ if $M$ is Riemannian or Lorentzian, respectively. Note that if $M$ is Riemannian or Lorentzian, then from \eqref{BiconservativeEq} we have $k_1+k_2=k_3$ or $k_1+k_2=-3k_3$, respectively. Thus, we have \eqref{TheoremCase4ODE} or \eqref{TheoremCase4ODE2}. Hence, we have obtained the case (iv) and the case (v) of the theorem.


\textbf{Case 3.} Let $\hat M$ be congruent to  the surfaces given in case (vi) of {Lemma} \ref{Case1IntSub1}. By a similar way to previous case we obtain the case surface given by \eqref{TheoremCase6} for a smooth function $f_2$. A direct computation shows that the principal curvatures of $M$ are $\displaystyle k_1=\frac{ f_2''}{\sqrt{(f_2'^2+1)^3}}$, $\displaystyle k_2=\frac{f_2'}{s\sqrt{(f_2'^2+1)}}$ and $\displaystyle k_3=\frac{1}{\sqrt{(f_2'^2+1)}}$  and $M$ is Riemannian. Since $M$ is Riemannian, we have  $k_1+k_2=k_3$ which gives the case (vi) of the theorem.


\textbf{Case 4.} Let $\hat M$ be congruent to  the surfaces given in case (vii) of {Lemma} \ref{Case1IntSub1}. In this case,  without loss of generality, we may assume $x(s_0,t,u)=y(t,u)$. Therefore, from \eqref{Case1CloseToClasssificationEq000} we have
$$\phi_1(s_0)\Theta_1(t)+\phi_2(s_0)\Theta_2(u)+\Gamma(s_0)=(At^2+Bu^2,t,u,At^2+Bu^2).$$
Thus, we get
\begin{subequations}\label{ClassTheoremEq3}
\begin{eqnarray}
\Theta_1(t)&=&\frac1{\phi_1(s_0)}(At^2+c_1,t+c_2,c_3,At^2+c_4),\\
\Theta_2(u)&=&\frac1{\phi_2(s_0)}(Bu^2+d_1,d_2,u+d_3,Bu^2+d_4)
\end{eqnarray}
\end{subequations}
for  some non-zero constants  $c_i,d_i$. By combining \eqref{ClassTheoremEq3} with \eqref{Case1CloseToClasssificationEq000} we obtain
\begin{equation}\label{ClassTheoremEq4}
x(s,t,u)=(A\psi_1 t^2+B\psi_2u^2,\psi_1 t,\psi_2u,A\psi_1 t^2+B\psi_2u^2)+\tilde\Gamma(s)
\end{equation}
for a smooth vector valued function $\tilde\Gamma=(\tilde\Gamma_1,\hdots,\tilde\Gamma_4)$ and some smooth functions $\psi_1(s),\psi_2(s)$. However, $\langle x_s,x_s\rangle=\varepsilon_1,\ \langle x_s,x_t\rangle=\langle x_s,x_u\rangle=0$ give
\begin{subequations}\nonumber
\begin{eqnarray}
\nonumber\psi_1'^2t^2+\psi_2'^2u^2+\langle\Gamma',\Gamma'\rangle+2(A\psi_1 t^2+B\psi_2u^2)(\tilde\Gamma_1'-\tilde\Gamma_4')+\psi_1\Gamma_2't+\psi_2\Gamma_3'u&=&\varepsilon_1,\\
\nonumber\psi_1t(\psi_1'+2A(\tilde\Gamma_4'-\tilde\Gamma_1'))+\psi_1\Gamma_2'&=&0,\\
\nonumber\psi_2u(\psi_2'+2B(\tilde\Gamma_4'-\tilde\Gamma_1'))+\psi_2\Gamma_3'&=&0.
\end{eqnarray}
\end{subequations}
Therefore, we have
$$(\tilde\Gamma_1-\tilde\Gamma_4)\neq0,\quad \psi_1=2A(\tilde\Gamma_1-\tilde\Gamma_4)+a_1,\quad \psi_2'=2B(\tilde\Gamma_1-\tilde\Gamma_4)+a_2$$
for some constants $a_1,a_2$ and, up to a suitable translation, we may assume $\Gamma_2=\Gamma_3=0$. Then,  we obtain a parametrization of $M$ as
\begin{align}\label{ClassTheoremEq6}
\begin{split}
x(s,t,u)=(&2(\tilde\Gamma_1-\tilde\Gamma_4)(A^2t^2+B^2u^2)+Aa_1t^2+Ba_2u^2+\tilde\Gamma_1,2A(\tilde\Gamma_1-\tilde\Gamma_4) t+a_1t,\\&2B(\tilde\Gamma_1-\tilde\Gamma_4) u+a_2u,2(\tilde\Gamma_1-\tilde\Gamma_4)(A^2t^2+B^2u^2)+Aa_1t^2+Ba_2u^2+\tilde\Gamma_4).
\end{split}
\end{align}
Note that if $AB=0$, then by a direct computation one can see that one of the principal curvatures vanishes identically on $M$. In this subcase, we obtain
$$ x( s, t, u)=\left(\frac 12  s t^2+s+\phi_1,s  t,\frac 12  s t^2+\phi_1, u\right)$$
which gives the case (ii) or case (iii) of the theorem. Thus, we assume $A\neq0, B\neq0$.

Next, we define new coordinates $(\bar s,\bar t,\bar u)$ such that $\bar s=\tilde\Gamma_1-\tilde\Gamma_4+a_1/2A$, $\bar t=2At$, $\bar u=2Bu$. From \eqref{ClassTheoremEq6} we obtain a parametrization of $M$ as given in \eqref{TheoremCase7} for a costant $a$ which is non-zero because $M$ has three distinct principal curvatures. Hence, we have the case (vii) of the theorem.

Hence, the proof of necessary condition is completed. The converse follows from a direct computation.
\end{proof}
\begin{Remark}
For the explicit parametrization of hypersurfaces given in case
(i)-(iii), see the complete classification of biconservative
surfaces in $\mathbb E^3$ and $\mathbb E^3_1$ which are given in
 \cite{HsanisField}, \cite{CMOP} and \cite{FuBiconservativeE31}.
\end{Remark}

\section*{Acknowledgements}
The first named author was supported by Project funded by China
Postdoctoral Science Foundation (No. 2014M560216) and the Excellent
Innovation talents Project of DUFE (No. DUFE2014R26).
The second named author  is supported by T\"UB\.ITAK  (Project Name: Y\_EUCL2TIP, Project Number: 114F199).


\end{document}